\documentclass[12pt,leqno]{article}

\usepackage{amsmath,amssymb,amsthm, amsfonts}
\usepackage[all]{xy}

\makeatletter

\newtheorem{theorem}{Theorem}[section]
\newtheorem{proposition}[theorem]{Proposition}
\newtheorem{lemma}[theorem]{Lemma}

\theoremstyle{definition}
\newtheorem{example}[theorem]{Example}
\newtheorem{definition}[theorem]{Definition}

\theoremstyle{remark}
\newtheorem{remark}[theorem]{\sc{Remark}}

	\title{Graphs and Ideals generated by some $2$-minors}
	\author{{\sc Masahiro Ohtani}}
	\date{\normalsize
Graduate School of Mathematics, Nagoya University\\
Chikusa-ku,  Nagoya 464--8602 JAPAN\\
{\small \tt m05011w@math.nagoya-u.ac.jp}}

\newcommand{\kk}[1]{\noindent {\rm (#1)}}

\let\Mathrm\operator@font

\def\standop#1{\mathop{\Mathrm #1}\nolimits}
\def\difstop#1#2{\expandafter\def\csname #1\endcsname{\standop{#2}}}
\def\defstop#1{\difstop{#1}{#1}}

\defstop{Ass}
\defstop{Assh}
\difstop{characteristic}{char}
\difstop{height}{ht}
\difstop{initial}{in}
\defstop{lcm}
\defstop{Spec}

\def\section{\@startsection{section}{1}{\z@ }%
{-3.5ex plus -1ex minus -.2ex}{2.3ex plus .2ex}{\bf }}

\long\def\refname{\par\kern -3ex
\begin{center}\rm R\sc{eferences}\end{center}\par\kern -2ex}

\def\@seccntformat#1{\csname the#1\endcsname.\quad}

\def\@@@sect#1#2#3#4#5#6[#7]#8{%
	\ifnum #2>\c@secnumdepth 
		\def \@svsec {}\else \refstepcounter {#1}%
		\def\@svsec{}
	\fi 
	\@tempskipa #5\relax 
	\ifdim \@tempskipa >\z@ 
		\begingroup #6\relax \@hangfrom {\hskip #3\relax 
		\@svsec}{\interlinepenalty \@M #8\par }\endgroup 
		\csname #1mark\endcsname {#7}
	\else 
	\def \@svsechd {#6\hskip #3\@svsec #8\csname #1mark\endcsname {#7}}
	\fi \@xsect {#5}}

\def\@@@startsection#1#2#3#4#5#6{%
	\if@noskipsec \leavevmode \fi \par \@tempskipa #4\relax \@afterindenttrue 
	\ifdim \@tempskipa <\z@ \@tempskipa -\@tempskipa \@afterindentfalse 
	\fi \if@nobreak \everypar {}\else \addpenalty {\@secpenalty }\addvspace 
		{\@tempskipa }\fi \@ifstar {\@ssect {#3}{#4}{#5}{#6}}{\@dblarg 
		{\@@@sect {#1}{#2}{#3}{#4}{#5}{#6}}}}

\def\theparagraph{\thesection.\arabic{paragraph}}
\def\aparagraph{\@@@startsection{paragraph}{2}{\z@ }%
              {1.75ex plus .2ex minus .15ex}{-1em}{\bf(\theparagraph) } }
\def\paragraph{\@@@startsection{paragraph}{2}{\z@ }%
              {1.75ex plus .2ex minus .15ex}{-1em}{}{\bf(\theparagraph)} }
\let\c@theorem\c@paragraph

\begin{document}

\maketitle

\begin{abstract}
 Let $G$ be a finite graph on $[n] = \{1,2,\ldots,n\}$,
$X$ a $2 \times n$ matrix of indeterminates over a field $K$ and
$S = K[X]$ a polynomial ring over $K$.
 In this paper,
we study about ideals $I_G$ of $S$ generated by $2$-minors $[i,j]$ of $X$
which correspond to edges $\{i,j\}$ of $G$.
 In particular, we construct a Gr\"obner basis of $I_G$
as a set of paths of $G$ and compute a primary decomposition.
\end{abstract}

\section{Introduction}

 Let $K$ be a field, $X = (X_{ij})$ an $m \times n$ generic matrix over $K$
and $r \le \min(m,n)$ a positive integer.
 The ideal $I_r(X)$ generated by {\it all} $r$-minors of $X$
in a polynomial ring $S = K[X_{ij} \mid 1 \le i \le m, 1 \le j \le n]$ is
called {\it the determinantal ideal} and
is studied by many researchers from many different viewpoints.
 For example, $I_r(X)$ is a prime ideal in $S$ and
the quotient ring $S/I_r(X)$ is a Cohen-Macaulay ring,
see Bruns and Herzog \cite{BH} or Bruns and Vetter \cite{BV}.

 In contrast, some kinds of ideals generated by {\it some} minors of $X$ have been thought.

 Conca defined ladder determinantal ideals in \cite{Con}.
 They are prime ideals and the quotient rings are Cohen-Macaulay.

 Diaconis, Eisenbud and Sturmfels studied the ideal generated all ``adjacent''
$2$-minors in a $2 \times n$ generic matrix \cite{DES}.
 An adjacent $2$-minor of a $2 \times n$ matrix 
is the determinant of a submatrix
with column indices $j$ and $j+1$ for $j = 1,2, \ldots, n-1$.
 This ideal is generated by a regular sequence and is not prime if $n>2$.
 They compute a primary decomposition and all minimal prime ideals of them.
 Ho\c{s}ten and Sullivant studied {\it ideals of adjacent minors}
as a generalization of ideals of Diaconis, Eisenbud and Sturmfels in \cite{HS}.
 An adjacent $r$-minor of $X$ is the determinant of a submatrix
with row indices $a_1, a_2, \ldots, a_r$ and column indices $b_1, b_2, \ldots, b_r$
where these indices are consecutive integers.
 They compute the minimal prime ideals of ideals of adjacent minors.
 
 Diaconis, Eisenbud and Sturmfels also found a minimal primary decomposition
of the ideal which is generated by all ``corner minors''  in \cite{DES}.
 A corner minor is the determinant of a $2 \times 2$ submatrix with
row indices $1, i$ and column indices $1, j$.

 In this paper,
we study the following ideals generated by some $2$-minors 
of a $2 \times n$ generic matrix.

 From now on, we use the following notation.
 Let 
$$ X = 
\begin{pmatrix}
	 X_1 & X_2 & \cdots & X_n \\
	 Y_1 & Y_2 & \cdots & Y_n
\end{pmatrix}$$
be a $2 \times n$ generic matrix over a field $K$
and $S = K[X_i,Y_i \mid 1 \le i \le n]$ a polynomial ring.
 For two integers $i, \, j \in [n] := \{1, 2, \ldots, n \}$,
we denote the $2$-minor
$$\det \begin{pmatrix}
	 X_i & X_j \\
	 Y_i & Y_j 
\end{pmatrix} = X_i Y_j - Y_i X_j$$
of $X$ by $[i,j]$.
 Let $G$ be a simple graph on $[n]$,
i.e., $G$ is a graph which does not have multiple edges or loops,
and we define an ideal $I_G$ as follows :
\footnote{After the first version of this paper was written, very recently,
it has been brought to my attention that
Herzog, Hibi, Hreinsd\'ottir, T. Kahle and  J. Rauh wrote a paper \cite{HHH}
which has considerable overlaps with this paper.}
	$$ I_G := \left( [i,j] \mid \text{$\{ i, j \}$ is an edge of $G$} \right). $$

 If $G$ is a path,
the ideal $I_G$ coincides with the ideal of adjacent minors of $X$.
 If $G$ is a star graph,
then the ideal $I_G$ coincides with the ideal of corner minors of $X$.
 So ideals $I_G$ are generalizations of these ideals.

 Properties of the ideal $I_G$ are closely connected to properties of the graph $G$.
 The aim of this paper is to describe a Gr\"obner basis 
and a primary decomposition of $I_G$ in relation to the datum of $G$.
 Particularly, we construct a Gr\"obner basis as ``a set of paths'' of $G$ and
describe an algorithm to compute a primary decomposition by ``operations'' of graphs.

 This paper is organized as follows.
 Section~2 is preliminaries.
 We define some properties of paths, e.g., minimality and irreducibility,
prove some facts and prepare some operations.
 In particular, decompositions of a path into a sum of irreducible paths
play important roles in Section~3.
 In Section~3, we calculate a Gr\"obner basis of $I_G$
as the set of all irreducible paths of $G$.
 In Section~4, we construct an ``algorithm'' of computation of a primary decomposition of $I_G$.
 In Section~5, we prove some results about connections between
some properties of ideals $I_G$ and graphs $G$.
 For example, we give a necessary condition for existence of Hamilton cycles of $G$.

\noindent{\sc Acknowledgment}.
 The author is grateful to Mitsuyasu Hashimoto, Yuhi Sekiya and Ken-ichi Yoshida for valuable conversations and helpful suggestions.
 He also expresses his thanks to the referee for his many pieces
of valuable advice.
 In particular, the example in Remark~\ref{countex.rmk} is due to him.

 The author was partially supported by JSPS Research Fellowships for Young Scientists.

\section{Preliminaries}

\paragraph\label{paths.par}
 A {\it walk} of $G$ is a sequence $p_0 e_0 p_1 e_1 \ldots e_\ell p_\ell$
satisfying that each $p_i$ is a vertex of $G$ for $i = 0, 1, \ldots, \ell$,
i.e., $p_i \in [n]$,
that each $e_i$ is an edge of $G$ which connects vertices $p_{i-1}$ and $p_i$ for $i = 1, 2, \ldots, \ell$
and that $p_0 \le p_\ell$.
 A walk $P = p_0 e_0 p_1 e_1 \ldots e_\ell p_\ell$ is a {\it path}
if additionally it holds $p_i \ne p_j$ for each indices $i \ne j$.
 In this case, we call $\ell$ the {\it length} of $P$.
 A walk $P = p_0 e_0 p_1 e_1 \ldots e_\ell p_\ell e_{\ell+1} p_0$ is
a {\it cycle} if the subsequence $p_0 e_0 p_1 e_1 \ldots e_\ell p_\ell$ is a path of $G$.

 For a path $P = p_0 e_1 p_1 e_2 \ldots e_\ell p_\ell$, 
we call an element of
	$V(P) = \{ p_i \mid 0 \le i \le \ell \}$
	(resp. $J(P) = \{ p_i \mid 0 < i < \ell \}$, and $E(P) = \{ p_0, p_\ell \}$)
a {\it vertex} (resp. a {\it joint}, and an {\it end}) of $P$.

\paragraph\label{abuse.par}
 Now $G$ is simple, so we can write $P = p_0 p_1 \ldots p_\ell$
short for a walk $P = p_0 e_0 p_1 e_1 \ldots e_\ell p_\ell$
with no confusion.

\paragraph\label{subwalk.par}
 Let $P = p_0 p_1 \ldots p_\ell$ be a walk.
 We take two vertices $p_i$ and $p_j$ with $i < j$.
 If $p_i < p_j$, we call the subsequence $p_i p_{i+1} \ldots p_j$
the {\it subwalk} of $P$ from $i$ to $j$, denoted $P_{i \to j}$.
 If $p_i > p_j$, the sequence $p_j p_{j-1} \ldots p_i$ is also called
the subwalk of $P$ from $j$ to $i$, denoted $P_{j \to i}$.

\paragraph\label{sum.par}
 For two paths $P = p_0 p_1 \ldots p_\ell$ and $Q = q_0 q_1 \ldots q_m$
with $\sharp E(P) \cap E(Q) = 1$,
we get the {\it sum} $P+Q$ of paths $P$ and $Q$.
 For example, assume that $p_0 = q_0$ and $p_\ell < q_m$.
 Then the sum $P+Q$ is the walk
	$p_\ell p_{\ell -1} \ldots p_0 q_1 \ldots q_{m-1} q_m$.

\paragraph\label{order.par}
 We define an order on the set of walks of $G$.
 For two walks $P = p_0 p_1 \ldots p_\ell$ and $Q = q_0 q_1 \ldots q_m$,
we say that $P \le Q$ if there is a sequence of indices
	$0 = j_0 < j_1 < \ldots < j_\ell = m$
such that $q_{j_i} = p_i$ for each $i = 0, 1, \ldots, \ell$.

 A walk $P = p_0 p_1 \ldots p_\ell$ is minimal with respect to this order
if and only if $P$ is a path
and there never exists an edge which connects two vertices of $P$
which do not adjoin.

\paragraph\label{irred.par}
 Let $P = p_0 p_1 \ldots p_\ell$ be a minimal path of $G$.
 We say that $P$ is {\it irreducible} if it holds that
	$J(P) \cap (p_0, p_\ell) = \emptyset$,
where $(p_0,p_\ell)$ is the subset $\{ x \mid p_0 < x < p_m \}$
of the set of real numbers $\mathbb{R}$.

\paragraph\label{def.gP.par}
 For an irreducible path $P = p_0 p_1 \ldots p_\ell$,
we set $g_P = M_P \cdot [p_0, p_\ell]$, where
$$ M_P = \prod _{p \in J(P)} Z_p, ~~~~~Z_p = 
\begin{cases}
	 Y_p & \text{if $p < p_0$} \\
	 X_p & \text{if $p > p_\ell$}
\end{cases}.$$

\paragraph\label{path.dec.par}
 Let $P = p_0 p_1 \ldots p_\ell$ be a minimal path of $G$.
 We can decompose $P$ into a sum of irreducible paths.

 We get a sequence $0 = i_0 < i_1 < \ldots < i_s = \ell$ as follows.
 Set $i_0 := 0$ and
$$ p_{i_t} := \begin{cases}
	 \min J(P_{i_{t-1} \to \ell}) \cap (p_{i_{t-1}}, p_\ell)
& \text{if $P_{i_{t-1} \to \ell}$ is not irreducible} \\
	 p_\ell & \text{if $P_{i_{t-1} \to \ell}$ is irreducible}
\end{cases}$$
for $t>0$.
 By definition, each subpath $P_{i_{t-1} \to i_t}$ is irreducible
and $P = P_{i_0 \to i_1} + P_{i_1 \to i_2} + \cdots + P_{i_{s-1} \to i_s}$ holds.

\paragraph\label{mon.div.par}
 Let $P = p_0 p_1 \ldots p_\ell$ be an irreducible path and
$p_a$ a vertex of $P$.
 Assume that $p_a > p_m$.
 Then we decompose the subpath $P_{0 \to a}$ (or $P_{m \to a}$) into
a sum $P_{i_0 \to i_1} + P_{i_1 \to i_2} + \cdots + P_{i_{s-1} + i_s}$
of irreducible paths, see \kk{\ref{path.dec.par}}.

\begin{lemma}\label{mon.div.lem}
 Under the notation in \kk{\ref{mon.div.par}},
each monomial $M_{P_{i_{t-1} \to i_t}}$ for $t = 1, 2, \ldots, s$,
and each variable $X_{p_{i_t}}$ for $t = 1, 2, \ldots, s-1$, divides $M_P$.
\end{lemma}

\begin{proof}
 It holds that $p_{i_t} > p_0$ by definition,
so $p_{i_t} > p_m$ holds since $P$ is irreducible.
 Then $X_{p_{i_t}}$ divides $M_P$.

 Set $Q = P_{i_{t-1} \to i_t}$ and take a joint $p_j$ of $Q$.
 It is enough to prove that $M_P$ is divided by the variable $Z_{p_j}$
if so is $M_Q$.
 If $p_j > p_{i_t} > p_m$, then the variable $Z_{p_j} = X_{p_j}$ divides
$M_Q$ and $M_P$.
 On the other hand, assume that $p_j < p_{i_{t-1}}$.
 By definition of $p_{i_{t-1}}$ (or earlier one),
$p_j$ must be less than $p_0$.
 Then the variable $Z_{p_j} = Y_{p_j}$ divides $M_Q$ and $M_P$.
\end{proof}

\paragraph\label{connect.par}
 Let $P = p_0 p_1 \ldots p_\ell$ and $Q = q_0 q_1 \ldots q_m$ be
irreducible paths of $G$ with $p_0 = q_0$ and $p_\ell < q_m$.
 Take a minimal path $R = r_0 r_1 \ldots r_k \le P+Q$.
 We decompose $R$ into a sum
	$R_{i_0 \to i_1} + R_{i_1 \to i_2} + \cdots + R_{i_{s-1} \to i_s}$
of irreducible paths of $G$, see \kk{\ref{path.dec.par}}.

\begin{lemma}\label{connect.lem}
 Under the notation in \kk{\ref{connect.par}},
each monomial $M_{R_{i_{t-1} \to i_t}}$ for $t = 1, 2, \ldots, s$, and
each variable $X_{r_{i_t}}$ for $t = 1, 2, \ldots, s-1$,
divides $\lcm(M_P,M_Q) \cdot Y_{p_0}$.
\end{lemma}

\begin{proof}
 It holds that $r_{i_t} > p_\ell > p_0 = q_0$ by definition,
so the variable $X_{r_{i_t}}$ divides $M_P$ or $M_Q$.

 We set $R' = R_{i_{t-1} \to i_t}$ and take a joint $r_j$ of $R'$.
 It is enough to prove that $\lcm (M_P,M_Q) \cdot Y_{p_0}$ is divided by
the variable $Z_{r_j}$ if so is $M_{R'}$.
 If $r_j > r_{i_t} > p_\ell$, then $r_j > p_0 = q_0$ and the variable
$Z_{r_j} = X_{r_j}$ divides $M_P$ or $M_Q$.
 On the other hand, assume that $r_j < r_{i_{t-1}}$.
 Then $r_j < r_{i_0} = p_\ell$ holds by definition of $r_{i_{t-1}}$
(or earlier one).
 So $r_j \le p_0 = q_0$ holds and the variable $Z_{r_j} = Y_{r_j}$ divides
$M_P$ or $M_Q$ or $Y_{p_0}$.
\end{proof}

\begin{remark}\label{reverse.rmk}
 We have analogues to discussions and results from \kk{\ref{path.dec.par}}
in the following situation.

 Let $P = p_0 p_1 \ldots p_\ell$ be a minimal path.
 We decompose $P$ into a sum
	  $P_{i_s \to i_{s-1}} + P_{i_{s-1} \to i_{s-2}} + \cdots + P_{i_1 \to i_0}$
of irreducible paths of $G$ as follows.

 Set $i_0 = \ell$ and we inductively define indices $i_t$ by
$$ p_{i_t} := \begin{cases}
	 \max J(P_{0 \to i_{t-1}}) \cap (p_0, p_{i_{t-1}})
& \text{if $P_{0 \to i_{t-1}}$ is not irreducible} \\
	 p_0 & \text{if $P_{0 \to i_{t-1}}$ is irreducible}
\end{cases}.$$

\kk{1}
 Let $P = p_0 p_1 \ldots p_\ell$ be an irreducible path of $G$
and $p_a$ a joint of $P$ with $p_a < p_0$.
 We decompose $P_{a \to 0}$ (or $P_{a \to m}$) into a sum
	 $P_{i_s \to i_{s-1}} + P_{i_{s-1} \to i_{s-2}} + \cdots + P_{i_1 \to i_0}$
of irreducible paths as above.
 Then each monomial $M_{P_{i_t \to i_{t-1}}}$ for $t = 1, 2, \ldots, s$, and
each variable $Y_{p_{i_t}}$ for $t = 1, 2, \ldots, s-1$, divides $M_P$.

\kk{2}
 Let $P = p_0 p_1 \ldots p_\ell$ and $Q = q_0 q_1 \ldots q_m$ be
irreducible paths of $G$ with $p_\ell = q_m$ and $p_0 < q_0$.
 Take a minimal path $R \le P+Q$.
 We decompose $R$ into a sum
	 $R_{i_s \to i_{s-1}} + R_{i_{s-1} \to i_{s-2}} + \cdots + R_{i_1 \to i_0}$
of irreducible paths as above.
 Then each monomial $M_{R_{i_t \to i_{t-1}}}$ for $t = 1, 2, \ldots, s$, and
each variable $Y_{r_{i_t}}$ for $t = 1, 2, \ldots, s-1$, divides
$\lcm (M_P, M_Q) \cdot X_{p_\ell}$.
\end{remark}

\section{Gr\"obner Basis}

\paragraph\label{refer.par}
 In this section, we use some definitions, properties and facts
about Gr\"obner basis,
e.g., monomial orders, definition of (reduced) Gr\"obner basis, and so on.
 We refer the reader to \cite{CLO} for more information on them.

 The following theorem is one of the main results of this paper.

\begin{theorem}\label{main.thm}
 The set
	 ${\cal G} = \{ g_P \mid \text{$P$ is an irreducible path of $G$} \}$
is the reduced Gr\"obner basis of $I_G$ with respect to
the reverse lexicographic order $< = <_{\rm revlex}$ on $S$ with
	 $Y_1 > Y_2 > \cdots > Y_n > X_1 > X_2 > \cdots > X_n$.
\end{theorem}
\begin{example}\label{goranger.ex}
 Let $G$ be the following graph :
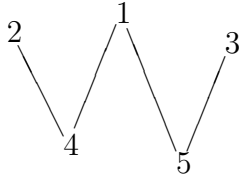
\begin{center}
\setlength\unitlength{0.5mm}
\begin{picture}(80,50)(0,0)
	\put(10,40){$2$}
	\put(25,10){$4$}
	\put(39,45){$1$}
	\put(55,5){$5$}
	\put(68,37){$3$}
	\put(25,15){\line(-1,2){12}}
	\put(28,17){\line(2,5){11}}
	\put(55,11){\line(-2,5){13}}
	\put(58,11){\line(2,5){10}}
\end{picture}
\end{center}
\vskip -5mm
 A path of $G$ is determined by its ends, so $G$ has ten paths.
 But the Gr\"obner basis ${\cal G}$ of $I_G$ consists of {\it nine} binomials
since the path $2$--$4$--$1$--$5$ is not irreducible.
 Explicitly,
$${\cal G} = \left\{
	\begin{array}{c}
		 [1,4], \, [1,5], \, [2,4], \, [3,5], \, Y_1[4,5], \\
		 X_4[1,2], \, X_5[1,3], \, Y_1X_5[3,4], \, Y_1X_4X_5[2,3]
	\end{array} \right\}$$
is the reduced Gr\"obner basis of $I_G$.
\end{example}

  For the rest of this section, we prove Theorem~\ref{main.thm}.

\paragraph\label{initial.par}
 For an irreducible path $P = p_0 p_1 \ldots p_\ell$,
the initial monomial $\initial_< g_P$ of $g_P$ is $M_P \cdot X_{p_0} Y_{p_\ell}$.

 Put
	 $J(P) = \{ p_{i_1} < p_{i_2} < \cdots < p_{i_s}
		 < p_{j_1} < p_{j_2} < \cdots < p_{j_t} \}$
with $p_{i_s} < p_0 < p_\ell < p_{j_1}$.
 Then the binomial $g_P$ is equal to
$$ \prod_{1 \le u \le s} Y_{p_{i_u}} \cdot X_{p_0} Y_{p_m} \cdot \prod_{i \le v \le t} X_{p_{j_v}}
	 - \prod_{1 \le u \le s} Y_{p_{i_u}} \cdot Y_{p_0} X_{p_m} \cdot \prod_{i \le v \le t} X_{p_{j_v}}.$$

 Two terms of an element of ${\cal G}$ cannot be divided
by the initial monomial of any element of ${\cal G}$,
so ${\cal G}$ is reduced if ${\cal G}$ is a Gr\"obner basis of $I_G$.

\paragraph\label{Buchberger.par}
 From now, we prove that ${\cal G}$ is a Gr\"obner basis of $I_G$
by Buchberger's criterion.

 In this and the next section,
let $R$ be a polynomial ring over a field $K$ with a monomial order $<$.
 For two polynomials $f$, $g \in R$, then
	 $$ S(f,g) = \frac{\initial_< g}{\gcd (\initial_< f, \initial_< g)} \cdot f - \frac{\initial_< f}{\gcd (\initial_< f, \initial_< g)} \cdot g$$
is called the $S$-polynomial of $f$ and $g$.

\begin{proposition}[Buchberger's criterion]\label{Buchberger.prop}
 Let $I$ be an ideal in $R$.
 A finite system of generators
	 ${\cal G} = \{ g_1, g_2, \ldots, g_t \} \subset I$
is a Gr\"obner basis for $I$
if and only if for all pairs $i \ne j$,
the remainder on division of the $S$-polynomial $S(g_i, g_j)$ by ${\cal G}$ is zero.
\end{proposition}

\begin{proof}
 See Cox, Little and O'Shea \cite{CLO}, Chapter 2, \S 6, Theorem~6.
\end{proof}

\paragraph\label{include.par}
 First, we prove that ${\cal G} \subset I_G$.

 Take an irreducible path $P = p_0 p_1 \ldots p_\ell$.
 We prove it by induction on the length $\ell$.
 If $\ell = 1$, $P$ is an edge of $G$ and
$g_P = [p_0, p_\ell]$ is contained in $I_G$.
 Assume that $\ell > 1$.
 Then $J(P)$ is not empty,
so either $J(P) \cap (-\infty, p_0)$ or $J(P) \cap(p_\ell, \infty)$ is not empty.
 If $J(P) \cap (-\infty, p_0) \ne \emptyset$ holds, we set
	 $p_a := \max J(P) \cap (-\infty, p_0)$.
 Then subpaths $P_{a \to 0}$ and $P_{a \to \ell}$ are irreducible and
	 $M_P = M_{P_{a \to 0}} \cdot M_{P_{a \to \ell}} \cdot Y_a$ holds.
 Then $g_P = Y_{p_0} M_{P_{a \to 0}} \cdot g_{P_{a \to \ell}} - Y_{p_\ell} M_{P_{a \to \ell}} \cdot g_{P_{a \to 0}}$ is contained in $I_G$.
 If $J(P) \cap(p_\ell, \infty) \ne \emptyset$, we can prove similarly for
	 $a = \min J(P) \cap(p_\ell, \infty)$.

\paragraph\label{class.par}
 We classify pairs of polynomials in ${\cal G}$
whose initial monomials are not coprime
whether their ends are the same or not.

 Let $P = p_0 p_1 \ldots p_\ell$ and $Q = q_0 q_1 \ldots q_m$
be irreducible paths of ${\cal G}$.
 Set $a = p_0$, $b = p_\ell$, $c = q_0$ and $d = q_m$.
 We suppose that $a \le c$ without loss of generality.
 In addition, we assume that $\gcd(\initial_< g_P, \initial_< g_Q) \ne 1$.

\paragraph\label{case1.par}
 We first suppose that $a = c$ and $b = d$.
 Then the $S$-polynomial $S(g_P, g_Q)$ is equal to zero,
so we have nothing to do.

\paragraph\label{case2.par}
 We secondly suppose that $a=c$ and $b \ne d$ hold.
 We can assume that $b<d$, then
	 $S(g_P,g_Q) = \lcm (M_P, M_Q) \cdot (Y_a X_b Y_d - Y_a Y_b X_d)$.
 We take a minimal path $R = r_0 r_1 \ldots r_k \le P+Q$ and
decompose $R$ into a sum
	 $R_{i_0 \to i_1} + R_{i_1 \to i_2} + \cdots + R_{i_{s-1} \to i_s}$
of irreducible paths as in \kk{\ref{connect.par}}.
 By Lemma~\ref{connect.lem},
	 $$ S(g_P, g_Q) = M' \cdot X_b X_{r_{i_1}} X_{r_{i_2}} \cdots X_{r_{i_{s-1}}} Y_d - Y_b X_{r_{i_1}} X_{r_{i_2}} \cdots X_{r_{i_{s-1}}} X_d,$$
where $M'$ is a monomial which can be divided by each $M_{R_{i_{t-1} \to i_t}}$ for $t = 1, 2, \ldots,s$.
 Then it holds that
\begin{align*}
 S(g_P,g_Q)
	 &\overset{\kk{s}}{\equiv}
		 M' \cdot (X_b X_{r_{i_1}} X_{r_{i_2}} \cdots X_{r_{i_{s-2}}} Y_{r_{i_{s-1}}} X_d-Y_b X_{r_{i_1}} X_{r_{i_2}} \cdots X_{r_{i_{s-1}}} X_d) \\
	 &\overset{\kk{s-1}}{\equiv}
		 M' \cdot (X_b X_{r_{i_1}} X_{r_{i_2}} \cdots Y_{r_{i_{s-2}}} X_{r_{i_{s-1}}} X_d-Y_b X_{r_{i_1}} X_{r_{i_2}} \cdots X_{r_{i_{s-1}}} X_d) \\
	 &\equiv \cdots \equiv 0,
\end{align*}
where the equivalence \kk{u} is induced by $g_{R_{i_{u-1} \to i_u}}$.
 As it were, the monomial $M_{R_{i_{t-1} \to i_t}}$ is a ``catalyst''
to exchange letters ``$X$'' and ``$Y$''.

 If $b=d$ and $a \ne c$, we can similarly prove using fact \kk{2}
in Remark~\ref{reverse.rmk}.

\paragraph\label{case3.par}
 Until the end of this section, we suppose that $a \ne c$ and $b \ne d$.
 In this case, the following may hold only if $b<c$ :
\begin{center}
	 \kk{a} the variable $Y_b$ divides $M_Q$,~~~~~~~~
	 \kk{b} the variable $X_c$ divides $M_P$.
\end{center}
 In fact, the condition \kk{a} (resp. \kk{b}) is equivalent to
that $b \in J(Q)$ and $b<c$ (resp. $c \in J(P)$ and $b<c$).

\paragraph\label{case3.1.par}
 In this paragraph, we suppose that neither \kk{a} nor \kk{b} hold.
 Set $M'_P = \frac{M_P}{H}$ and $M'_Q = \frac{M_Q}{H}$,
where $H = \gcd (M_P, M_Q)$.
 Then,
\begin{align*}
	 S(g_P,g_Q) &= H M'_P M'_Q (Y_a X_b X_c Y_d - X_a Y_b Y_c X_d) \\
		 &= M'_P Y_a X_b \cdot g_Q - M'_Q Y_c X_f \cdot g_P.
\end{align*}

\paragraph\label{case3.2.par}
 In this paragraph, we suppose that \kk{a} holds and \kk{b} does not hold.
 Now $b$ is a joint of $Q$ and $b<c$.
 So we decompose subpaths $Q_{b \to c}$ and $Q_{b \to d}$ into sums
\begin{gather*}
	 Q_{b \to c} = Q_{i_0 \to i_1} + Q_{i_1 \to i_2} + \cdots + Q_{i_{s-1} \to i_s}, \\
	 Q_{b \to d} = Q_{j_0 \to j_1} + Q_{j_1 \to j_2} + \cdots + Q_{j_{t-1} \to j_t}.
\end{gather*}
of irreducible paths of $G$ as in \kk{\ref{path.dec.par}}.
 Set $M'_P = \frac{M_P}{H}$ and $M'_Q = \frac{M_Q}{Y_b H}$,
where $H = \gcd (M_P,M_Q)$.
 Then,
\begin{align}
\notag	 S(g_P,g_Q) &= M'_P X_a \cdot g_Q - M'_Q X_c Y_d \cdot g_P \\
\tag{$\clubsuit$}	 &= H M'_P M'_Q (Y_a X_b X_c Y_d - X_a Y_b Y_c X_d).
\end{align}
 The first monomial of \kk{$\clubsuit$} can be written as
$$ N_1 \cdot \prod_{1 \le v \le t} M_{Q_{j_{v-1} \to j_v}} \cdot X_b Y_{j_1} Y_{j_2} \cdots Y_{j_{t-1}} Y_d,$$
where $N_1$ is a suitable monomial.
 Each $M_{Q_{j_{v-1} \to j_v}}$ is a ``catalyst'' to exchange
letters ``$X$'' and ``$Y$'',
so it is equivalent to $H M'_P M'_Q \cdot Y_a Y_b X_c X_d$.
 Similarly, the second monomial of \kk{$\clubsuit$} is equivalent to
the same monomial by $g_{Q_{i_{u-1} \to i_u}}$'s.

 Using fact \kk{1} in Remark~\ref{reverse.rmk},
$S(g_P,g_Q)$ is equivalent to $0$ by ${\cal G}$
if \kk{a} does not hold and \kk{b} holds.

\paragraph\label{case3.3.par}
 In this paragraph, we suppose that \kk{a} and \kk{b} holds.
 Now $b$ (resp. $c$) is a joint of $Q$ (resp. $P$).
 Put $M'_P = \frac{M_P}{H X_c}$ and $M'_Q = \frac{M_Q}{H Y_b}$,
where $H = \gcd (M_P, M_Q)$.
 Then $S(g_P,g_Q) = M'_Q Y_d g_P - M'_PX_a g_Q
	 = H M'_P M'_Q (X_a Y_b Y_c X_d - Y_a X_b X_c Y_d)$.
 We decompose the subpath $P_{a \to c}$ (resp. $Q_{b \to d}$) into a sum
	 $P_{i_0 \to i_1} + P_{i_1 \to i_2} + \cdots + P_{i_{s-1} \to i_s}$
	 (resp. $Q_{j_t \to j_{t-1}} + Q_{j_{t-1} \to j_{t-2}} + \cdots + Q_{j_1 \to j_0}$)
of irreducible paths of $P$ as in \kk{\ref{mon.div.par}}
(resp. in Remark~\ref{reverse.rmk} \kk{1}).
 By Lemma~\ref{mon.div.lem} and the fact \kk{1} in Remark~\ref{reverse.rmk},
the monomials $M_{P_{i_{u-1}\to i_u}}$ and the variables $X_{p_{i_u}}$
(resp. $M_{Q_{j_v \to j_{v-1}}}$ and $Y_{j_{i_v}}$) divide $M_{P}$ (resp. $M_Q$).

 The $S$-polynomial $S(g_P,g_Q)$ is written as
\begin{gather}
\tag{$\spadesuit$}
	 M_1 \cdot X_a X_{i_1} X_{i_2} \cdots X_{i_{s-1}} Y_c - M_2 \cdot X_b Y_{j_{t-1}} Y_{j_{t-2}} \cdots Y_{i_1} Y_d,
\end{gather}
where $M_1$ (resp. $M_2$) is a monomial which is divided by
	 $\prod _{1 \le u \le s} M_{P_{i_{u-1} \to i_u}}$
(resp. $\prod _{1 \le v \le t} M_{Q_{j_v \to j_{v-1}}}$).
 Each monomial in $\kk{\spadesuit}$ is equivalent to the monomial
$H M'_P M'_Q Y_a Y_b X_c X_d$ by $g_{P_{i_{u-1} \to i_u}}$'s and $g_{Q_{j_v \to j_{v-1}}}$'s,
so its remainder with respect to ${\cal G}$ is zero.

\section{Primary Decomposition}

\begin{proposition}\label{reduced.prop}
 The ideal $I_G$ is a radical ideal.
\end{proposition}

 To prove Proposition~\ref{reduced.prop}, the following lemma is essential.

\begin{lemma}\label{easy.lem}
 Let $I$ be an ideal in a polynomial ring $R$ over a field $K$.
 Assume that the initial ideal $\initial_< I$ with respect to a monomial
order $<$ is generated by squarefree monomials.
 Then $I$ is a radical ideal.
\end{lemma}

\begin{proof}
 Note that $\initial_< I$ is a radical ideal.
 Assume that $\sqrt{I} \ne I$ and take $f \in \sqrt{I} \setminus I$.
 Taking a normal form of $f$, we can suppose that
	 $\initial_< f \not\in \initial_< I$.
 Take an integer $n$ with $f^n \in I$.
 Then $(\initial_< f)^n = \initial_< f^n \in \initial_< I$,
so $\initial_< f \in \initial_< I$ and it is a contradiction.
\end{proof}

 By Proposition~\ref{reduced.prop}, the ideal $I_G$ is the intersection of
all prime ideals which contain $I_G$.
 In this section, we study a way to find all minimal prime ideals of $I_G$.

\paragraph\label{norm.par}
 Denote $d_G$ the distance on the set of vertices of $G$
which is defined by lengths of paths, namely, for two vertices $x$, $y$ of $G$,
	 $$ d_G (x,y) := \min \{ n \mid \text{there is a path $P$ of length $n$ with $E(P) = \{ x,y \}$} \}, $$
or $d_G (x,y) = \infty$ if there is not such a path.

\paragraph\label{neighbor.par}
 Let $v$ be a vertex of $G$.
 We call the set $\{ x \mid d_G(x,v) = 1 \}$ the {\it neighborhood} of $v$ in $G$,
denoted $N_G(v)$.
 We say that $G$ is {\it complete around $v$} if it holds that
$d_G(x,y) \le 1$ for any $x, y \in N_G(v)$.

\paragraph\label{subdet.ideal.par}
 For a subset $A \subset [n]$, we denote the ideal
	 $([a,b] \mid a, b \in A)$
in $S$ by $I_2(A)$.

\begin{proposition}\label{prime.prop}
 The following conditions are equivalent : \\
\kk{1} $G$ is complete around all vertices of $G$, \\
\kk{2} $G$ is a disjoint union of complete graphs, \\
\kk{3} the ideal $I_G$ is a prime ideal.
\end{proposition}

\begin{proof}
\kk{1}$\Rightarrow$\kk{2}.
 Let $x, y$ be vertices of $G$ with $x<y$ and $d_G(x,y) < \infty$.
 It is enough to prove that $d_G(x,y) = 1$.
 Assume that $t := d_G(x,y) \ge 2$ and
take a path $P = p_0 p_1 \ldots p_t$ of length $t$ with $p_0 = x$ and $p_t = y$.
 $G$ is complete around $p_1$, so there is an edge $\{ p_0, p_2 \}$ and
it contradicts $d_G(x,y) = t$.\\
\kk{2}$\Rightarrow$\kk{3}
 Put $G = \coprod _{1 \le i \le a} G_i$,
where each $G_i$ is a complete graph.
 We denote $V_i \subset [n]$ the vertex set of $G_i$.
 Then $I_G = \sum_{i} I_2(V_i)$ and
$$ S/I_G \simeq \bigotimes _i \frac{k[X_v, Y_v \mid v \in V_i]}{I_2(V_i)}.$$
 Each $k[X_v, Y_v \mid v \in V_i]/I_2(V_i)$ is a determinantal ring,
so $S/I_G$ is an integral domain.\\
$\kk{3}\Rightarrow\kk{1}$.
 Assume that there is a vertex $v$ around which $G$ is not complete.
 Then there are vertices $u,w \in N_G(v)$ with $d_G(u,w)=2$.
 Then the equation
\begin{gather}\label{key.eq} \tag{$\heartsuit$}
	 0 = \det \begin{pmatrix}
		 X_u & X_v & X_w \\
		 X_u & X_v & X_w \\
		 Y_u & Y_v & Y_w
	 \end{pmatrix} = X_u \cdot [v,w] - X_v \cdot [u,w] + X_w \cdot [u,v]
\end{gather}
implies that $X_v \cdot [u,w] \in I_G$.
 $I_G$ is not prime since $X_v \not\in I_G$ and $[u,w] \not\in I_G$.
\end{proof}

\paragraph\label{step.par}
 In this paragraph, we suppose that $I_G$ is not prime.
 By Proposition~\ref{prime.prop}, there is a vertex $v$
around which $G$ is not complete.
 Then there are $u, w \in N_G(v)$ with $d_G(u,w) = 2$.

 Let $P$ be a prime ideal of $S$ which contains $I_G$.
 By the equation~\kk{\ref{key.eq}},
it holds that $X_v \in P$ or that $[u,w] \in P$.
 Similarly it holds that $Y_v \in P$ or that $[u,v] \in P$,
so $(X_v, Y_v) \subset P$ or $[u,w] \in P$ holds.
 Thinking all pairs $(u,w)$ of $N_G(v)$,
$P$ contains one of the following ideals :
	$$ \kk{1}~~I_G + (X_v, Y_v), ~~~~~~~~~~~~~~~
		 \kk{2}~~I_G + I_2(N_G(v)). $$

 These ideals correspond to the following operations of graphs :
\begin{gather} \tag{$\diamondsuit$}
\begin{array}{l}
	 \text{\kk{1} taking away $v$ and all edges of which $v$ is an end from $G$,} \\
	 \text{\kk{2} adding all edges which connect two vertices in $N_G(v)$ to $G$.}
\end{array}
\end{gather}

\begin{lemma}\label{step.lem}
 Let $v$ be a vertex around which $G$ is not complete, then
$$ I_G = \left( I_G + (X_v, Y_v) \right)
		 \cap \left( I_G + I_2(N_G(v)) \right).$$
\end{lemma}

\begin{proof}
 The subsumption $\subset$ is trivial.
 Let $P$ be a prime ideal which contains $I_G$.
 Then $P$ contains the right hand side by the above discussion, 
so it holds that the right hand side is contained in $\sqrt{I_G} = I_G$.
\end{proof}

\begin{lemma}\label{decrease.lem}
 In the above operations \kk{$\diamondsuit$},
	 the number of vertices around which the graph is not complete decreases.
\end{lemma}

\begin{proof}
 Set $G_1$ (resp. $G_2$) as the graph which is made by the operation 
\kk{1} (resp. \kk{2}).
 Take a vertex $u \ne v$ of $G$.
 It is easy to see that $G$ is not complete around $u$ if so is $G_1$.

 Assume that $G_2$ is not complete around $u$.
 Then there are $x, y \in N_{G_2}(u)$ with $d_{G_2}(x,y) = 2$.
 We have nothing to prove if $d_G(x,y) = 2$.
 Otherwise, an edge $\{x, u\}$ or $\{y,u\}$ is added by the operation,
i.e., it holds that $u \in N_G(v)$ and that $x$ or $y \in N_G(v)$.
 If both $x$ and $y \in N_G(v)$, then $d_{G_2}(x,y) = 1$.
 So $N_G(v)$ does not contain $x$ or $y$.
 If $x \not\in N_G(v)$, then $x$ and $v \in N_G(u)$ and $d_G(x,v) = 2$.
 Then $G$ is not complete around $u$.
\end{proof}

 By Lemma~\ref{step.lem} and Lemma~\ref{decrease.lem},
repeating operations \kk{$\diamondsuit$},
we have a set of prime ideals which contains all minimal prime ideals of $I_G$.
 So we can decompose $I_G$ into the intersection of some prime ideals.

\begin{example}\label{D5.ex}
  Let $G$ be the following graph :
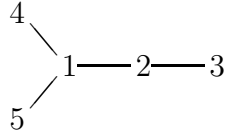
\begin{center}
\setlength\unitlength{0.7mm}
\begin{picture}(40,20)(0,0)
	 \put(0,0){$5$}
	 \put(0,20){$4$}
	 \put(10,10){$1$}
	 \put(24,10){$2$}
	 \put(38,10){$3$}
	 \put(4,4){\line(5,6){5}}
	 \put(4,20){\line(5,-6){5}}
	 \put(13,12){\line(1,0){10}}
	 \put(27,12){\line(1,0){10}}
\end{picture}
\end{center}
 Now, $G$ is not complete around $1$.
 Then we get the following graphs :
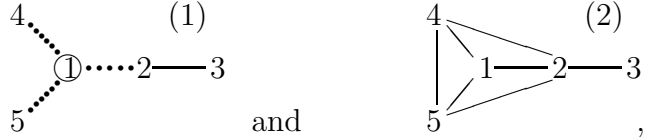
\begin{center}
\setlength\unitlength{0.7mm}
\begin{picture}(40,25)(0,0)
	 \put(30,20){\kk{1}}
	 \put(0,0){$5$}
	 \put(0,20){$4$}
	 \put(10,10){$1$}
	 \put(11,12){\circle{5}}
	 \put(24,10){$2$}
	 \put(38,10){$3$}
	 \multiput(4,4)(1,1){6}{\circle*{1}}
	 \multiput(4,20)(1,-1){6}{\circle*{1}}
	 \multiput(15,12)(2,0){5}{\circle*{1}}
	 \put(27,12){\line(1,0){10}}
\end{picture}~~ and ~~~~~~~~~~
\begin{picture}(40,25)(0,0)
	 \put(30,20){\kk{2}}
	 \put(0,0){$5$}
	 \put(0,20){$4$}
	 \put(10,10){$1$}
	 \put(24,10){$2$}
	 \put(38,10){$3$}
	 \put(2,5){\line(0,1){14}}
	 \put(4,4){\line(5,6){5}}
	 \put(4,3){\line(3,1){20}}
	 \put(4,20){\line(5,-6){5}}
	 \put(4,21){\line(3,-1){20}}
	 \put(13,12){\line(1,0){10}}
	 \put(27,12){\line(1,0){10}}
\end{picture},
\end{center}
where a circled number means a vertex which is taken away.
 The left graph \kk{1} is a disjoint union of complete graphs,
but the right graph \kk{2} is not complete around the vertex $2$.
 So we operate on the graph \kk{2} for \kk{$\diamondsuit$},
we get the following graphs.
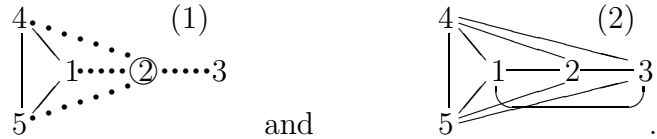
\begin{center}
\setlength\unitlength{0.7mm}
\begin{picture}(40,25)(0,0)
	 \put(30,20){\kk{1}}
	 \put(0,0){$5$}
	 \put(0,20){$4$}
	 \put(10,10){$1$}
	 \put(25,12){\circle{5}}
	 \put(24,10){$2$}
	 \put(38,10){$3$}
	 \put(2,5){\line(0,1){14}}
	 \put(4,4){\line(5,6){5}}
	 \multiput(4,3)(3,1){7}{\circle*{1}}
	 \put(4,20){\line(5,-6){5}}
	 \multiput(4,21)(3,-1){7}{\circle*{1}}
	 \multiput(13,12)(2,0){5}{\circle*{1}}
	 \multiput(29,12)(2,0){5}{\circle*{1}}
\end{picture} ~~ and ~~~~~~~~~~
\begin{picture}(40,25)(0,0)
	 \put(30,20){\kk{2}}
	 \put(0,0){$5$}
	 \put(0,20){$4$}
	 \put(10,10){$1$}
	 \put(24,10){$2$}
	 \put(38,10){$3$}
	 \put(2,5){\line(0,1){14}}
	 \put(4,4){\line(5,6){5}}
	 \put(4,3){\line(3,1){20}}
	 \put(4,2){\line(4,1){32}}
	 \put(25,9){\oval(28,8)[b]}
	 \put(4,20){\line(5,-6){5}}
	 \put(4,21){\line(3,-1){20}}
	 \put(4,22){\line(4,-1){32}}
	 \put(13,12){\line(1,0){10}}
	 \put(27,12){\line(1,0){10}}
\end{picture}.
\end{center}
 Both graphs are disjoint unions of complete graphs.
 So the process finishes and we get the decomposition
	 $$ I_G = (X_1,Y_1,[2,3]) \cap (X_2,Y_2, [1,4], [1,5], [4,5]) \cap I_2(X), $$
where $I_2(X)$ is the ideal generated by all $2$-minors of $X$.
\end{example}

\section{Properties of Graphs}

 In this section, we suppose that $G$ is a connected simple graph on $[n]$
with $n > 2$.

\begin{definition}\label{Hamilton.def}
 A path $P$ of $G$ is {\it Hamilton} if $V(P) = [n]$.
 A cycle $p_0 p_1 \ldots p_\ell p_0$ is {\it Hamilton}
if the subpath $p_0 p_1 \ldots p_\ell$ is a Hamilton path.
\end{definition}

 In graph theory,
it is a very difficult and important problem 
whether G has a Hamilton cycle.
 In fact, a nontrivial necessary and sufficient condition for existence of
Hamilton cycles is not known.

 For example, $G$ has a Hamilton cycle if one of the following holds :\\
\kk{1} $\sharp N_G(a) + \sharp N_G(b) \ge n$ for any vertices $a$ and $b$, 
where $\sharp$ means the number of the elements of the set, see Ore~\cite{Ore}.\\
\kk{2} $G$ is planar and $4$-connected, see Tutte~\cite{Tutte}.

\begin{proposition}\label{ham.path.prop}
 If $G$ has a Hamilton path, then $\height I_G = n-1$.
 The converse holds if $G$ is a tree, i.e., $G$ has no cycle.
\end{proposition}

\begin{proof}
 Permuting vertices, we have a graph $G' \simeq G$ which has a Hamilton
path $1$--$2$--$\cdots$--$n$. Then $I_G'$ contains a sequence
 $[1,2], \, [2,3], \, \ldots, \, [n-1,n]$.
 It is a regular sequence and $\height I_G = \height I_{G'} \ge n-1$.
 It holds that $\height I_G \le n-1$ since $I_G \subset I_2(X)$.

 Assume that $G$ is a tree.
 $G$ has a vertex $v$ with $\sharp N_G(v) \ge 3$
if $G$ does not have a Hamilton path.
 Applying the operation~\kk{1} of \kk{$\diamondsuit$} to $v$,
we get a graph $\tilde{G}$ which has at most $n-4$ edges.
 Then $\height I_{\tilde{G}} \le n-4$, so $\height I_G \le \height I_{\tilde{G}} + 2 \le n-2$.
\end{proof}

\begin{remark}\label{hidora.ex}
 If $G$ is not a tree,
the converse of Proposition~\ref{ham.path.prop} is not true in general.
 For example, let $G$ be the following graph :

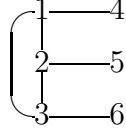
\begin{center}
\setlength\unitlength{1mm}
\begin{picture}(30,15)
  \put(10,14){$1$}
	\put(10,7){$2$}
	\put(10,0){$3$}
	\put(20,14){$4$}
	\put(20,7){$5$}
	\put(20,0){$6$}
	\put(12,15){\line(1,0){8}}
	\put(12,8){\line(1,0){8}}
	\put(12,1){\line(1,0){8}}
	\put(10,8){\oval(6,14)[l]}
	\put(11,3){\line(0,1){4}}
	\put(11,10){\line(0,1){4}}
\end{picture}
\end{center}
\vskip -2mm
 Then $\height I_G = 5$ but $G$ does not have a Hamilton path.
\end{remark}

\paragraph\label{assh.par}
 For an ideal $I$ in $S$, we denote the set of prime ideals
	 $\{ P \in \Spec S \mid P \supset I \, \text{and}\, \height I = \height P \}$
by $\Assh_S{S/I}$.

\begin{proposition}\label{ham.cycle.prop}
 If $G$ has a Hamilton cycle, then $\Assh_S {S/I_G} = \{ I_2(X) \}$.
\end{proposition}

\begin{proof}
 Take $P \in \Assh_S {S/I_G}$.
 Then $P$ is generated by some variables and some $2$-minors of $X$,
see \kk{\ref{step.par}}.
 We suppose that $P$ contains a variable $X_v$.
 Then $G \setminus v$ is a graph of size $n-1$ and has a Hamilton path,
so $\height I_{G \setminus v} = n-2$.
 $P$ contains $I_{G \setminus v} + (X_v, Y_v)$, then
	 $\height P \ge \height I_{G \setminus v} + 2 = n$,
it contradicts $\height P = \height I_G = n-1$.
\end{proof}

\begin{remark}\label{countex.rmk}
 The converse of Proposition~\ref{ham.cycle.prop} is not true in general.
 For example, let $G$ be the following graph :
 \vskip 2mm
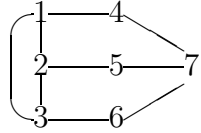
\begin{center}
\setlength\unitlength{1mm}
\begin{picture}(40,15)
  \put(10,14){$1$}
	\put(10,7){$2$}
	\put(10,0){$3$}
	\put(20,14){$4$}
	\put(20,7){$5$}
	\put(20,0){$6$}
	\put(30,7){$7$}
	\put(12,15){\line(1,0){8}}
	\put(12,8){\line(1,0){8}}
	\put(12,1){\line(1,0){8}}
	\put(10,8){\oval(6,14)[l]}
	\put(11,3){\line(0,1){4}}
	\put(11,10){\line(0,1){4}}
	\put(22,15){\line(5,-3){8}}
	\put(22,8){\line(1,0){8}}
	\put(22,1){\line(5,3){8}}
\end{picture}
\end{center}
 Now $G$ has Hamilton paths,
so $\height I_G = 6$ holds and $I_2(X)$ is contained in $\Assh_S {S/I_G}$.
 However $G \setminus 7$ is isomorphic to the graph in Remark~\ref{hidora.ex},
and $G \setminus v$ has Hamilton paths for each $v = 1, 2, \ldots, 6$.
 Then we have $\height I_{G \setminus v} = 5$ for each $v = 1, 2, \ldots, 7$.
 So $\height P \ge 7$ holds for each associated prime ideal $P \ne I_2(X)$.
\end{remark}

\end{document}